\documentclass[12pt]{amsbook}
\usepackage[all,cmtip]{xy}

\usepackage{a4wide}
\usepackage{pxfonts}
\usepackage{wasysym}
\usepackage{amssymb,amsxtra,mathrsfs}
\usepackage{eucal}
\usepackage{setspace,upgreek}
\usepackage{mathtools}
\usepackage{url}
\usepackage[hidelinks]{hyperref}

\newcommand{\mf}[1]{\mathfrak{#1}}
\newcommand{\mc}[1]{\mathcal{#1}}

\newcommand{\mr}[1]{\mathrm{#1}}

\newcommand{\B}{\mathcal{B}}

\newcommand{\lr}[2]{\langle #1,#2\rangle}
\newcommand{\w}[1]{\widetilde{#1}}

\newcommand{\up}{\upharpoonright}

\newcommand{\ze}{\mathbf{0}} 
\newcommand{\un}{\mathbf{1}} 
\newcommand{\N}{\mathbb{N}}
\newcommand{\K}{\mathbb{K}}
\newcommand{\R}{\mathbb{R}}

\newcommand{\C}{\mathbb{C}}
\newcommand{\Z}{\mathbb{Z}}

\newtheorem{theorem}{Theorem}
\newtheorem{definition}[theorem]{Definition}

\newtheorem{corollary}[theorem]{Corollary}
\newtheorem{notation}[theorem]{Notation}

\theoremstyle{definition}

\newtheorem{casi notevoli}[theorem]{Casi notevoli}
\newtheorem{introduction}
[theorem]{Introduction}

\newtheorem{remark}
[theorem]
{Remark}

\begin{document}

\title[Stokes Equalities for Forms Functions of an Unbounded Spectral Operator]
{Stokes-type Integral Equalities for Scalarly Essentially Integrable Locally Convex Vector Valued Forms
which are Functions of an Unbounded Spectral Operator}

\author{Benedetto Silvestri}

\email{bndsilvestri@proton.me}

\address{440 W Kelly Ave Apt 6 Jackson, WY 83001-8544 USA (via express couriers signature required)}

\urladdr{\url{https://arxiv.org/a/silvestri_b_1}}

\urladdr{https://orcid.org/0000-0002-8042-4932}


\begin{abstract}
In this work we establish a Stokes-type integral equality for scalarly essentially integrable forms on an orientable
smooth manifold with values in the locally convex linear space $\langle B(G),\sigma(B(G),\mathcal{N})\rangle$,
where $G$ is a complex Banach space and $\mathcal{N}$ is a suitable linear subspace of the norm dual of $B(G)$. 
This result widely extends the Newton-Leibnitz-type equality stated in one of our previous articles.
To obtain our equality we generalize the main result of that article, and employ the Stokes theorem for smooth locally
convex vector valued forms established in a prodromic paper.
Two facts are remarkable. Firstly the forms integrated involved in the equality are functions of a possibly unbounded
scalar type spectral operator in $G$. Secondly these forms need not be smooth nor even continuously differentiable. 
\end{abstract}
\maketitle
\begin{introduction}
In this work we establish in Thm. \ref{09071617}
a Stokes-type integral equality for scalarly essentially integrable $\lr{B(G)}{\sigma(B(G),\mc{N})}$-valued forms
on an orientable smooth manifold, where $G$ is a complex Banach space. 
This result widely extends the Newton-Leibnitz-type equality established in \cite[Cor. 2.33]{SilInt}.
To obtain the equality we employ the Extension Thm. \ref{18051958ta} a generalization of \cite[Thm. 2.25]{SilInt}
along with the Stokes theorem for smooth locally convex vector valued forms \cite[Thm.3.59]{silscess}.
Two facts are remarkable. Firstly these forms are functions of a possibly unbounded scalar type spectral operator in $G$.
Secondly these forms need not be smooth nor even continuously differentiable. 
\end{introduction}
\begin{notation}
In the present work we employ the notation of \cite{SilInt} and these of \cite{silscess}, with the following two
remarks. 
First what in \cite{SilInt} is called ``Radon measure'' and meant measure in the sense of Bourbaki
\cite[Ch. $III$, $\S1$, $n^{\circ}3$, Def. 2]{IntBourb}, here accordingly will be called simply ``measure''.
Second if $Z$ is a $\K$-locally convex vector space with $\K\in\{\R,\C\}$, then
we let $Z^{\prime}=\mc{L}(Z,\K)$ denote the topological dual of $Z$.
\par
If $G$ is a $\C$-Banach space, then let $\mr{ClO}(G)$ denote the set of closed operators in $G$.
If $X$ is a locally compact space and $\mu$ is a measure on $X$, then a map $f:X\to\C$
is scalarly essentially $\mu$-integrable or simply essentially $\mu$-integrable iff 
$\mf{R}\circ\imath_{\C}^{\C_{\R}}\circ f$ and $\mf{I}\circ\imath_{\C}^{\C_{\R}}\circ f$
are essentially $\mu$-integrable, where $\mf{R},\mf{I}\in\mc{L}(\C_{\R},\R)$ are the real and imaginary part
respectively.
\par
We recall from \cite[pg. 39-40]{SilInt} that if $\lr{Z}{\tau}$ is a Hausdorff locally convex space over
$\K\in\{\R,\C\}$, then by definition $f:X\to\lr{Z}{\tau}$
is scalarly essentially $(\mu,Z)$-integrable, or $f:X\to Z$ is scalarly essentially $(\mu,Z)$-integrable with respect
to the topology $\tau$, iff $\uppsi\circ f$ is essentially $\mu$-integrable for every
$\uppsi\in\lr{Z}{\tau}^{\prime}$ and the weak integral of $f$ belongs to $Z$, namely there exists a necessarily unique
element $s\in Z$ such that $\uppsi(s)=\int(\uppsi\circ f)\,d\mu$ for every $\uppsi\in\lr{Z}{\tau}^{\prime}$.
In such a case we shall define $\int f\,d\mu\coloneqq s$.
\par
Let $N\in\Z_{+}^{\ast}$, define $\mr{P}^{[N]}:\R^{N}\to\R^{N-1}$, $x\mapsto x\up[1,N-1]\cap\Z$ if $N>1$; 
$x\mapsto 0$ if $N=1$.
Let $M$ be a nonzero dimensional manifold with boundary and let $(U,\phi)$ be a boundary chart of $M$, define 
$\phi^{\partial M}\coloneqq(\mr{P}^{[\mr{dim}\,M]}\circ\imath_{\phi(U)}^{\R^{\mr{dim}\,M}}\circ
\phi\circ\imath_{U\cap\partial M}^{U})_{\natural}$,
where $f_{\natural}=f\up^{\mr{Range}(f)}$ for any map $f$.
Let $\mc{U}$ be a collection of charts of $M$,
and let $\mc{U}_{\partial}$ be the subcollection of those elements in $\mc{U}$ that are boundary charts,
define $\mc{U}^{\partial}\coloneqq\{(U\cap\partial M,\phi^{\partial M})\,\vert\,(U,\phi)\in\mc{U}_{\partial}\}$.
If $\mc{U}$ is an atlas of $M$, then $\mc{U}^{\partial}$ is an atlas of $\partial M$,
moreover if $M$ is oriented and $\mc{U}$ is oriented, then $\mc{U}^{\partial}$ is oriented
and $(U\cap\partial M,\phi^{\partial M})$ is $\gamma$-oriented iff
$(U,\phi)\in\mc{U}$ is $\gamma$-oriented, with $\gamma\in\{1,-1\}$.
\par
We fix the following data.
A $\C$-Banach space $G$; a possibly \textbf{unbounded} scalar type spectral operator $R$ in $G$,
let $\sigma(R)$ be its spectrum and let $E$ be its resolution of identity;
an $E-$appropriate set $\mc{N}$ \cite[Def.2.11]{SilInt};
a scalar type spectral operator $T\in B(G)$ and let $\sigma(T)$ denote its spectrum;
locally compact spaces $X,Y$ and measures $\mu$ and $\nu$ on $X$ and $Y$ respectively;
a finite dimensional smooth manifold $M$, with or without boundary, such that $N\coloneqq\mr{dim}\,M\neq 0$.
\end{notation}
\begin{theorem}
\label{18051958taPRE}
Let $\{\sigma_{n}\}_{n\in\N}$ be an $E-$sequence,
let the maps $X\ni x\mapsto f_{x}\in \mr{Bor}(\sigma(R))$ and $Y\ni y\mapsto u_{y}\in \mr{Bor}(\sigma(R))$
be such that $\widetilde{f_{x}}\in\mathfrak{L}_{E}^{\infty}(\sigma(R))$, $\mu-l.a.e.(X)$
and $\widetilde{u_{y}}\in\mathfrak{L}_{E}^{\infty}(\sigma(R))$, $\nu-l.a.e.(Y)$.
Let $X\ni x\mapsto f_{x}(R)\in\lr{B(G)}{\sigma(B(G),\mc{N})}$ and $Y\ni y\mapsto u_{y}(R)\in\lr{B(G)}{\sigma(B(G),\mc{N})}$
be scalarly essentially $(\mu,B(G))-$integrable and $(\nu,B(G))-$integrable respectively, while let
$g,h\in\mr{Bor}(\sigma(R))$. If for all $n\in\N$, 
\begin{equation}
\label{19181001ta}
g(R_{\sigma_{n}}\up G_{\sigma_{n}})
\int\,f_{x}(R_{\sigma_{n}}\up G_{\sigma_{n}})\,d\,\mu(x)
\subseteq
h(R_{\sigma_{n}}\up G_{\sigma_{n}})
\int\,u_{y}(R_{\sigma_{n}}\up G_{\sigma_{n}})\,d\,\nu(y),
\end{equation}
then 
\begin{equation}
\label{16241402ta}
g(R)\int\,f_{x}(R)\,d\,\mu(x)\up\Theta
=
h(R)\int\,u_{y}(R)\,d\,\nu(y)\up\Theta.
\end{equation}
In 
\eqref{19181001ta} the weak-integrals are with respect to the measures $\mu$ and $\nu$
and with respect to the $\sigma(B(G_{\sigma_{n}}),\mc{N}_{\sigma_{n}})$-topology, while in \eqref{16241402ta}
\begin{equation*}
\Theta
\doteqdot
\mr{Dom}\left(g(R)\int\,f_{x}(R)\,d\,\mu(x)\right)
\cap
\mr{Dom}\left(h(R)\int\,u_{y}(R)\,d\,\nu(y)\right),
\end{equation*}
and the weak-integrals are with respect to the measures $\mu$ and $\nu$ and with respect to the
$\sigma(B(G),\mc{N})$-topology.
\end{theorem}
\begin{proof}
\eqref{19181001ta} is meaningful by \cite[Thm. 2.22]{SilInt}. By \cite[(1.18)]{SilInt}, for all $z\in\Theta$
\begin{alignat}{2}
\label{19051223ta}
g(R)\int f_{x}(R)\,d\,\mu(x)\,z
&=
\lim_{n\in\N}E(\sigma_{n})g(R)\int f_{x}(R)\,d\,\mu(x)\,z
\notag
\\
\intertext{by \cite[Thm. 18.2.11(g)]{ds} and \cite[(2.25)]{SilInt}}
&=
\lim_{n\in\N}g(R)\int f_{x}(R)\,d\,\mu(x)\,E(\sigma_{n})z
\notag
\\
\intertext{by\cite[(2.31)]{SilInt} and \cite[Lemma 1.7]{SilInt} applied to $g(R)$}
&=
\lim_{n\in\N}
g(R_{\sigma_{n}}\up G_{\sigma_{n}})
\int 
f_{x}(R_{\sigma_{n}}\up G_{\sigma_{n}})
\,d\,\mu(x)
\,
E(\sigma_{n})
z
\notag
\\
\intertext{by hypothesis \eqref{19181001ta}}
&=
\lim_{n\in\N}
h(R_{\sigma_{n}}\up G_{\sigma_{n}})
\int\,u_{y}(R_{\sigma_{n}}\up G_{\sigma_{n}})\,d\,\nu(y)
\,
E(\sigma_{n})
z
\notag
\\
\intertext{by what above proven and by replacing $g$ with $h$, $f$ with $u$ and $\mu$ with $\nu$}
&=
h(R)\int\,u_{y}(R)\,d\,\nu(y)\,z.
\end{alignat}
\end{proof}
\begin{theorem}
[\textbf{$\sigma(B(G),\mc{N})-$Extension Theorem}]
\label{18051958ta}
Let $X\ni x\mapsto f_{x}\in\mr{Bor}(\sigma(R))$ be such that
$\widetilde{f_{x}}\in\mf{L}_{E}^{\infty}(\sigma(R))$, $\mu-l.a.e.(X)$
and $X\ni x\mapsto f_{x}(R)\in\lr{B(G)}{\sigma(B(G),\mc{N})}$ be scalarly essentially $(\mu,B(G))-$integrable.
Moreover let $Y\ni y\mapsto u_{y}\in\mr{Bor}(\sigma(R))$ be such that
$\widetilde{u_{y}}\in\mf{L}_{E}^{\infty}(\sigma(R))$, $\nu-l.a.e.(Y)$ and
$Y\ni y\mapsto u_{y}(R)\in\lr{B(G)}{\sigma(B(G),\mc{N})}$ be scalarly essentially $(\nu,B(G))-$integrable.
Finally let $g,h\in\mr{Bor}(\sigma(R))$ and assume that
\footnote{For instance but not necessarily when $\w{h}\in\mf{L}_{E}^{\infty}(\sigma(R))$ since in such a case
\cite[Thm. 18.2.11]{ds} implies $h(R)\in B(G)$}
\begin{equation}
\label{09042001}
h(R)\int\,u_{y}(R)\,d\,\nu(y)\in B(G).
\end{equation}
If $\{\sigma_{n}\}_{n\in\N}$ is an $E-$sequence and for all $n\in\N$
\begin{equation}
\label{19211001ta}
g(R_{\sigma_{n}}\up G_{\sigma_{n}})
\int\,f_{x}(R_{\sigma_{n}}\up G_{\sigma_{n}})\,d\,\mu(x)
\subseteq
h(R_{\sigma_{n}}\up G_{\sigma_{n}})
\int\,u_{y}(R_{\sigma_{n}}\up G_{\sigma_{n}})\,d\,\nu(y),
\end{equation}
then
\begin{equation}
\label{16011401ta}
g(R)\int\,f_{x}(R)\,d\,\mu(x)
=
h(R)\int\,u_{y}(R)\,d\,\nu(y).
\end{equation}
In \eqref{19211001ta} the weak-integral are with respect to the measures $\mu$ and $\nu$
and with respect to the $\sigma(B(G_{\sigma_{n}}),\mc{N}_{\sigma_{n}})$-topology,
while in \eqref{16011401ta} the weak-integral is with respect to the measures $\mu$ and $\nu$
and with respect to the $\sigma(B(G),\mc{N})$-topology.
\end{theorem}
Notice that $g(R)$ and $h(R)$ are possibly \textbf{unbounded} operators in $G$.
\begin{proof}
\eqref{09042001} and \eqref{16241402ta} imply
\begin{equation}
\label{W15051936ta}
g(R)\int\,f_{x}(R)\,d\,\mu(x)\subseteq h(R)\int\,u_{y}(R)\,d\,\nu(y).
\end{equation}
Let us set
\begin{equation}
\label{17502703ta}
(\forall n\in\N)
(\delta_{n}\doteqdot\overset{-1}{|g|}([0,n])).
\end{equation}
We claim that
\begin{equation}
\label{17542703ta}
\begin{cases}
\bigcup_{n\in\N}\delta_{n}=\sigma(R)\\
n\geq m\Rightarrow\delta_{n}\supseteq\delta_{m}\\
(\forall n\in\N)(g(\delta_{n})\text{ is bounded.})
\end{cases}
\end{equation}
Since $|g|\in \mr{Bor}(\sigma(R))$ we have $\delta_{n}\in\B(\C)$ for all $n\in\N$, so $\{\delta_{n}\}_{n\in\N}$ is an
$E-$sequence, hence by \cite[(1.18)]{SilInt}
\begin{equation}
\label{10511201ta}
\lim_{n\in\N}E(\delta_{n})=\un;
\end{equation}
with respect to the strong operator topology on $B(G)$. Indeed the first equality follows by
\begin{equation*}
\bigcup_{n\in\N}\delta_{n}\doteq\bigcup_{n\in\N}\overset{-1}{|g|}([0,n])=\overset{-1}{|g|}\left(\bigcup_{n\in\N}[0,n]\right)
=\overset{-1}{|g|}(\R^{+})=\mr{Dom}(g)\doteqdot\sigma(R),
\end{equation*}
the second by the fact that $\overset{-1}{|g|}$ preserves the inclusion, the third by the inclusion
$|g|(\delta_{n})\subseteq [0,n]$. Hence our claim. By the third statement of \eqref{17542703ta},
$\delta_{n}\in\B(\C)$ and \cite[Lemma 1.7(3)]{SilInt} we obtain
\begin{equation}
\label{19292703ta}
(\forall n\in\N)(E(\delta_{n})G\subseteq\mr{Dom}(g(R))).
\end{equation}
By \cite[(2.25)]{SilInt} and \eqref{19292703ta} for all $n\in\N$
\begin{equation*}
\int\,f_{x}(R)\,d\,\mu(x)E(\delta_{n})G\subseteq E(\delta_{n})G\subseteq\mr{Dom}(g(R)).
\end{equation*}
Therefore
\begin{equation*}
(\forall n\in\N)(\forall v\in G)\left(E(\delta_{n})v\in\mr{Dom}
\left(g(R)\int\,f_{x}(R)\,d\,\mu(x)\right)\right).
\end{equation*}
Hence by \eqref{10511201ta}
\begin{equation}
\label{WII04041215ta}
\mathbf{D}\doteqdot\mr{Dom}\left(g(R)\int\,f_{x}(R)\,d\,\mu(x)\right)\text{ is dense in }G.
\end{equation}
Now
$\int\,f_{x}(R)\,d\,\mu(x)\in B(G)$ and $g(R)$ is closed by \cite[Thm. 18.2.11]{ds}, so by \cite[Lemma 1.15]{SilInt}
we find that 
\begin{equation}
\label{WII04041228ta}
g(R)\int\,f_{x}(R)\,d\,\mu(x)\text{ is closed.}
\end{equation}
Next \eqref{09042001} and \eqref{W15051936ta} imply
\begin{equation}
\label{WII04041231ta} 
g(R)\int\,f_{x}(R)\,d\,\mu(x)\in B(\mathbf{D},G).
\end{equation}
Now \eqref{WII04041228ta}, \eqref{WII04041231ta} and \cite[Lemma 1.16]{SilInt} imply that $\mathbf{D}$ is closed in $G$,
therefore by \eqref{WII04041215ta}
\begin{equation*}
\mathbf{D}=G;
\end{equation*}
therefore the statement follows by \eqref{W15051936ta}. 
\end{proof}
\begin{definition}
Let $V$ be an open neighbourhood of $\sigma(R)$, $l\in\R_{+}^{\ast}\cup\{+\infty\}$ such that $]-l,l[\cdot V\subseteq V$,
and $F:V\to\C$ be analytic. Moreover let $W$ be a set and $g:W\to\R$ such that $g(W)\subseteq]-l,l[$.
Let $F_{t}:V\ni\lambda\mapsto F(t\lambda)\in\C$ with $t\in]-l,l[$, then define 
the following operator valued map originating by the Borel functional calculus of the operator $R$
\begin{equation*}
\zeta_{F,g}^{R}:W\ni x\mapsto F_{g(x)}(R)\in\mr{ClO}(G).
\end{equation*}
\end{definition}
\begin{corollary}
\label{09051744}
Let $V$ be an open neighbourhood of $\sigma(T)$, $l\in\R_{+}^{\ast}\cup\{+\infty\}$ such that $]-l,l[\cdot V\subseteq V$,
and $F:V\to\C$ be analytic. Moreover let $n,p\in\Z_{+}^{\ast}$, $W$ be an open set of $\R^{n}$, and $g\in\mc{C}^{p}(W,\R)$
such that $g(W)\subseteq]-l,l[$. Thus $\zeta_{F,g}^{T}\in\mc{C}^{p}(W,B(G))$, and for every $i\in[1,n]\cap\Z$ we have 
\begin{equation*}
\frac{\partial\zeta_{F,g}^{T}}{\partial e_{i}}=\frac{\partial g}{\partial e_{i}}\cdot T\zeta_{\frac{dF}{d\lambda},g}^{T}.
\end{equation*}
\end{corollary}
\begin{proof}
$]-l,l[\ni t\mapsto F_{t}(T)\in B(G)$ is smooth since \cite[Thm. 1.21]{SilInt}, therefore
the first sentence of the statement follows since composition of $\mc{C}^{p}$-maps is a $\mc{C}^{p}$-map,
while the equality follows by the Chain Rule and by \cite[Thm. 1.21]{SilInt}.
\end{proof}
\begin{definition}
\label{09070820}
Let $k\in\Z_{+}$, $\omega\in\mr{Alt}_{c}^{k}(M)$ and $(U,\phi:U\to W)$ be a chart of $M$.
Define
\begin{equation*}
\begin{cases}
\omega^{\phi}:M(k,N,<)\to\mc{A}(W);
\\
I\mapsto(\imath_{U}^{M})^{\ast}(\omega)(\partial_{I_{1}}^{\phi},\dots,\partial_{I_{k}}^{\phi})\circ\phi^{-1}.
\end{cases}
\end{equation*}
Moreover let $V$ be an open neighbourhood of $\sigma(R)$ such that $\R\cdot V\subseteq V$, $F:V\to\C$ be analytic and
let $\delta\in\mc{B}(\C)$ be such that\footnote{for instance when $\delta$ is bounded see Rmk. \ref{09071645}}
\begin{equation}
\label{09182010}
\begin{aligned}  
\mr{Range}(\zeta_{F,\omega_{I}^{\phi}}^{R_{\delta}\up G_{\delta}})&\subseteq B(G_{\delta});
\\
\zeta_{F,\omega_{I}^{\phi}}^{R_{\delta}\up G_{\delta}}&\in
\mf{L}_{c}^{1}(W,\lr{B(G_{\delta})}{\sigma(B(G_{\delta}),\mc{N}_{\delta})},\lambda). 
\end{aligned}
\end{equation}
Define 
\begin{equation*}
\mr{f}_{\omega,\phi}^{\delta,F}:M(k,N,<)\ni I\mapsto\mr{f}_{\omega,\phi,I}^{\delta,F}
\coloneqq\zeta_{F,\omega_{I}^{\phi}}^{R_{\delta}\up G_{\delta}}\circ\phi,
\end{equation*}
and then define
$[\omega,\phi,\delta,F]\in\mr{Alt}^{k}(U,M;\lr{B(G_{\delta})}{\sigma(B(G_{\delta}),\mc{N}_{\delta})},\lambda)$ such that
\begin{equation*}
[\omega,\phi,\delta,F]
\coloneqq\sum_{I\in M(k,N,<)}\mr{f}_{\omega,\phi,I}^{\delta,F}\otimes\bigwedge_{s=1}^{k}dx_{I_{s}}^{\phi}.
\end{equation*}
\end{definition}
\begin{remark}
\label{09071645}
Let $k\in\Z_{+}$, $\omega\in\mr{Alt}_{c}^{k}(M)$ and $(U,\phi:U\to W)$ be a chart of $M$.
Let $\sigma\in\mc{B}(\C)$ be \emph{bounded}, thus $R_{\sigma}\up G_{\sigma}\in B(G_{\sigma})$ since \cite[Lemma 1.7]{SilInt}.
Hence $\zeta_{F,\omega_{I}^{\phi}}^{R_{\sigma}\up G_{\sigma}}\in\mc{A}_{c}(W,\lr{B(G_{\sigma})}{\|\cdot\|})$
by Cor. \ref{09051744}, so $\mr{f}_{\omega,\phi,I}^{\sigma,F}\in\mc{A}_{c}(U,\lr{B(G_{\sigma})}{\|\cdot\|})$
and then $[\omega,\phi,\sigma,F]$ is smooth w.r.t. the norm topology, namely
$[\omega,\phi,\sigma,F]\in\mr{Alt}^{k}(U,M;\lr{B(G_{\sigma})}{\|\cdot\|})$. 
Finally as a result $\zeta_{F,\omega_{I}^{\phi}}^{R_{\sigma}\up G_{\sigma}}$ is norm continuous and compactly supported,
therefore $\zeta_{F,\omega_{I}^{\phi}}^{R_{\sigma}\up G_{\sigma}}$ is Lebesgue integrable w.r.t. the norm topology and its integral
belongs to $B(G_{\sigma})$.
\end{remark}
\begin{remark}
\label{09071646}
Let $\delta\in\mc{B}(\C)$. The norm topology on $B(G_{\delta})$ is stronger than the topology
$\sigma(B(G_{\delta}),\mc{N}_{\delta})$ since this last is the weakest topology on $B(G_{\delta})$ among those
for which $\mc{N}_{\delta}$ is a set of continuous functionals, and since $\mc{N}_{\delta}\subseteq B(G_{\delta})^{\prime}$.
Thus we can and shall identify $\mc{A}(U,\lr{B(G_{\delta})}{\|\cdot\|})$ as a
$\mc{A}(U)$-submodule of $\mc{A}(U,\lr{B(G_{\delta})}{\sigma(B(G_{\delta}),\mc{N}_{\delta})})$ and
$\mr{Alt}^{k}(U,M;\lr{B(G_{\delta})}{\|\cdot\|})$ as a $\mc{A}(U)$-submodule of
$\mr{Alt}^{k}(U,M;\lr{B(G_{\delta})}{\sigma(B(G_{\delta}),\mc{N}_{\delta})})$.
\end{remark}
\begin{remark}
\label{09080624}
Let $\delta\in\mc{B}(\C)$, then any map defined on $X$ and with values in $B(G_{\delta})$ that is scalarly essentially 
$\mu$-integrable w.r.t. the norm topology it is also scalarly essentially $\mu$-integrable w.r.t. the
$\sigma(B(G_{\delta}),\mc{N}_{\delta})$-topology since $\mc{N}_{\delta}\subseteq B(G_{\delta})^{\prime}$.
\end{remark}
\begin{definition}
\label{09211309}
Let $k\in\Z_{+}$, $\omega\in\mr{Alt}_{c}^{k}(M)$ and $\{(U_{\alpha},\phi_{\alpha})\}_{\alpha\in D}$
be an atlas of $M$. Let $V$ be an open neighbourhood of $\sigma(R)$ such that $\R\cdot V\subseteq V$, $F:V\to\C$
be analytic and $\delta\in\mc{B}(\C)$ be such that \eqref{09182010} holds for $\phi=\phi_{\alpha}$ and for every
$\alpha\in D$. Define
$[\omega,\delta,F]\in\mr{Alt}^{k}(M;\lr{B(G_{\delta})}{\sigma(B(G_{\delta}),\mc{N}_{\delta})},\lambda)$
such that for all $\alpha\in D$
\begin{equation*}
(\imath_{U_{\alpha}}^{M})^{\times}([\omega,\delta,F])=[\omega,\phi_{\alpha},\delta,F].
\end{equation*}
\end{definition}  
\begin{definition}
\label{09211311}
Let $k\in\Z_{+}^{\ast}$, $\omega\in\mr{Alt}^{k-1}(M)$ and $i\in[1,k]\cap\Z$.
Define $\mf{d}_{i}(\omega)\in\mc{A}(M)$ and $\mf{n}_{i}(\omega)\in\mr{Alt}^{k}(M)$ such that
for any given atlas $\mc{U}$ of $M$ we have for every $(U,\phi)\in\mc{U}$
\begin{equation*}
\begin{aligned}
(\imath_{U}^{M})^{\ast}(\mf{d}_{i}(\omega))&\coloneqq
\partial_{i}^{\phi}[(\imath_{U}^{M})^{\ast}(\omega)
(\partial_{1}^{\phi},\dots,\widehat{\partial_{i}^{\phi}},\dots,\partial_{k}^{\phi})],
\\
(\imath_{U}^{M})^{\ast}(\mf{n}_{i}(\omega))&\coloneqq
(\imath_{U}^{M})^{\ast}(\omega)
(\partial_{1}^{\phi},\dots,\widehat{\partial_{i}^{\phi}},\dots,\partial_{k}^{\phi})
\bigwedge_{s=1}^{k}dx_{s}^{\phi};
\end{aligned}
\end{equation*}
where $\widehat{z}$ stands for $z$ missing.
\end{definition}
The above two definitions are well-set since the usual gluing lemma for smooth forms, 
since the extension of the gluing lemma via charts at scalarly essentially integrable locally convex vector valued
maps \cite[Rmk.1.20]{silscess}, and
since the extension of the gluing lemma via charts at smooth locally convex vector valued maps \cite[Notation 0.1]{silscess},
where the compatibility in both the definitions is ensured by the following simple fact 
\begin{equation*}
(\imath_{U_{\alpha}}^{M})^{\ast}(\omega)(\partial_{I_{1}}^{\phi_{\alpha}},\dots,\partial_{I_{k}}^{\phi_{\alpha}})
\circ\imath_{U_{\alpha,\beta}}^{U_{\alpha}}
=
(\imath_{U_{\alpha,\beta}}^{M})^{\ast}(\omega)
(\partial_{I_{1}}^{\phi_{\alpha,\beta}},\dots,\partial_{I_{k}}^{\phi_{\alpha,\beta}}),
\end{equation*}
where $U_{\alpha,\beta}=U_{\alpha}\cap U_{\beta}$ and
$\phi_{\alpha,\beta}=(\phi_{\alpha}\circ\imath_{U_{\alpha,\beta}}^{U_{\alpha}})_{\natural}$.
\begin{theorem}
[\textbf{Stokes equality for $\sigma(B(G),\mc{N})$-integrable forms functions of an unbounded operator}]  
\label{09071617}
Let $M$ be oriented with boundary and $\omega\in\mr{Alt}_{c}^{N-1}(M)$.
Let $V$ be an open neighbourhood of $\sigma(R)$ such that $\R\cdot V\subseteq V$ and $F:V\to\C$ be analytic. 
Assume that there exists a finite family $\{(U_{\alpha},\phi_{\alpha})\}_{\alpha\in D}$
of oriented charts of $M$ such that $\{U_{\alpha}\}_{\alpha\in D}$ is a covering of $\mr{supp}(\omega)$ and
\begin{enumerate}
\item
$\widetilde{F}_{t}\in\mf{L}_{E}^{\infty}(\sigma(R))$ for every $t\in\R$, and for all $\alpha\in D$ such that $\phi_{\alpha}$
is a boundary chart, the map
\begin{equation*}
\zeta_{F,\omega_{N}^{\phi_{\alpha}}}^{R}\circ\imath_{\phi_{\alpha}(U_{\alpha}\cap\partial M)}^{\phi_{\alpha}(U_{\alpha})}:
\phi_{\alpha}(U_{\alpha}\cap\partial M)\to\lr{B(G)}{\sigma(B(G),\mc{N})},
\end{equation*}
is scalarly essentially $(\lambda_{\phi_{\alpha}(U_{\alpha}\cap\partial M)},B(G))$-integrable,
\label{09071617h1}
\item
$(\widetilde{\frac{dF}{d\lambda}})_{t}\in\mf{L}_{E}^{\infty}(\sigma(R))$ for every $t\in\R$, and for all $\alpha\in D$ and 
$i\in[1,N]\cap\Z$, the map
\begin{equation*}
\zeta_{\frac{dF}{d\lambda},\omega_{i}^{\phi_{\alpha}}}^{R}:\phi_{\alpha}(U_{\alpha})\to\lr{B(G)}{\sigma(B(G),\mc{N})},
\end{equation*}
is scalarly essentially $(\lambda_{\phi_{\alpha}(U_{\alpha})},B(G))$-integrable;
\label{09071617h2}
\end{enumerate}
where
\begin{equation*}
\omega_{i}^{\phi_{\alpha}}\coloneqq
(\imath_{U_{\alpha}}^{M})^{\ast}(\omega)
(\partial_{1}^{\phi_{\alpha}},\dots,\widehat{\partial_{i}^{\phi_{\alpha}}},\dots,\partial_{N}^{\phi_{\alpha}})\circ\phi_{\alpha}^{-1}.
\end{equation*}
Thus
\begin{equation*}
R\int\sum_{i=1}^{N}(-1)^{i-1}\mf{d}_{i}(\omega)\cdot[\mf{n}_{i}(\omega),\sigma(R),\frac{dF}{d\lambda}]
=
\int(\imath_{\partial M}^{M})^{\times}([\omega,\sigma(R),F]);
\end{equation*}
where the integrals belong to $B(G)$ and are with respect to the $\sigma(B(G),\mc{N})$ topology and in case
$\partial M=\emptyset$, then the integral in the right-hand side has to be understood equal to $\ze$.
\end{theorem}
\begin{remark}
\label{09211852}
Let $\mc{U}=\{(U_{\alpha},\phi_{\alpha})\}_{\alpha\in D}$ and $\mc{U}^{\partial}$ be as in Notation.
Thus $\mc{U}^{\partial}$ is a family of oriented charts of $\partial M$ such that $\{Q_{\alpha}\}_{\alpha\in D}$,
with $Q_{\alpha}=U_{\alpha}\cap\partial M$ for every $\alpha\in D$, is a collection of open sets of $\partial M$
and a covering of $\mr{supp}(\omega)\cap\partial M$ compact set of $\partial M$.
Next set $D^{\dagger}=D\cup\{\dagger\}$, $U_{\dagger}=\complement_{M}\mr{supp}(\omega)$,
$Q_{\dagger}=\complement_{\partial M}(\mr{supp}(\omega)\cap\partial M)$ and 
let $\{\psi_{\alpha}\}_{\alpha\in D^{\dagger}}$ be a smooth partition of unity of $M$
subordinate to $\{U_{\alpha}\}_{\alpha\in D^{\dagger}}$
and $\{k_{\alpha}\}_{\alpha\in D^{\dagger}}$ be a smooth partition of unity of $\partial M$ subordinate to
$\{Q_{\alpha}\}_{\alpha\in D^{\dagger}}$. Thus since \eqref{09241810} and \eqref{09211740} applied to $\delta=\sigma(R)$
the statement of Thm. \ref{09071617} reads as follows
\begin{multline*}
R
\int
\sum_{\alpha\in D}
\gamma_{\phi_{\alpha}}
(\imath_{U_{\alpha}}^{M}\circ\phi_{\alpha}^{-1})^{\ast}(\psi_{\alpha})
\sum_{i=1}^{N}(-1)^{i-1}\frac{\partial\omega_{i}^{\phi_{\alpha}}}{\partial e_{i}}
\zeta_{\frac{dF}{d\lambda},\omega_{i}^{\phi_{\alpha}}}^{R}\,d\lambda_{\phi_{\alpha}(U_{\alpha})} 
=
\\
\int
\sum_{\alpha\in D}
\gamma_{\phi_{\alpha}}
\left(\imath_{U_{\alpha}\cap\partial M}^{\partial M}\circ(\phi_{\alpha}^{\partial M})^{-1}\right)^{\ast}(k_{\alpha})
\left(\zeta_{F,\omega_{N}^{\phi_{\alpha}}}^{R}\circ\imath_{\phi_{\alpha}(U_{\alpha}\cap\partial M)}^{\phi_{\alpha}(U_{\alpha})}
\circ(\mf{i}_{\R^{N-1}}^{\R^{N}}\circ\imath_{\phi_{\alpha}^{\partial M}(U_{\alpha}\cap\partial M)}^{\R^{N-1}})_{\natural}\right)\,
d\lambda_{\phi_{\alpha}^{\partial M}(U_{\alpha}\cap\partial M)};
\end{multline*}
where $\mf{i}_{\R^{N-1}}^{\R^{N}}:\R^{N-1}\to\R^{N}$ is such that if $N>1$, then
$\Pr_{k}^{\R^{N}}\circ\mf{i}_{\R^{N-1}}^{\R^{N}}=\Pr_{k}^{\R^{N-1}}$ if $k\in[1,N-1]\cap\Z$, and
$\Pr_{N}^{\R^{N}}\circ\mf{i}_{\R^{N-1}}^{\R^{N}}=\ze_{\R^{N-1}}$ the constant map on $\R^{N-1}$ equal to $0$;
while $\mf{i}_{\R^{0}}^{\R^{1}}:0\to 0$. Notice that 
$(\mf{i}_{\R^{N-1}}^{\R^{N}}\circ\imath_{\phi_{\alpha}^{\partial M}(U_{\alpha}\cap\partial M)}^{\R^{N-1}})_{\natural}$
is a diffeomorphism of $\phi_{\alpha}^{\partial M}(U_{\alpha}\cap\partial M)$ onto $\phi_{\alpha}(U_{\alpha}\cap\partial M)$
thus the right-hand side of the above equality is well-set since hypothesis \eqref{09071617h1} and the theorem of
change of variable in multiple integrals.
\end{remark}
\begin{remark}
The strategy employed to obtain Thm. \ref{09071617} is as follows:
Given an $E$-sequence of bounded sets $\{\sigma_{n}\}_{n\in\N}$ we apply for every $n\in\N$ the Stokes Thm.
for locally convex vector-valued forms \cite[Thm.3.59]{silscess} to the
$\lr{B(G_{\sigma_{n}})}{\sigma(B(G_{\sigma_{n}}),\mc{N}_{\sigma_{n}})}$-valued form
$[\omega,\sigma_{n},F]$ which is smooth as a result of Rmk. \ref{09071645}. 
Then develop the terms of these equalities by employing the families of oriented charts $\mc{U}$ and $\mc{U}^{\partial}$, 
and the families of smooth maps $\{\psi_{\alpha}\}_{\alpha\in D}$ and $\{k_{\alpha}\}_{\alpha\in D}$.
Finally we apply the Extension Thm. \ref{18051958ta} to the sequence of the resulting equalities.
\end{remark}
\begin{remark}
\label{09191021}
Thm. \ref{09071617} establishes a Stokes-type equality for $\lr{B(G)}{\sigma(B(G),\mc{N})}$-valued
integrable forms: (1) that arise from the Borelian functional calculus of the possibly \textbf{unbounded} operator $R$;
(2) that might be \textbf{not} smooth nor even continuously differentiable. 
To this regard we notice that the rigidity of analytic functions prevents any reasonable attempt to use 
the strong operator derivability on $\mr{Dom}(R)$ in \cite[Thm. 1.23(2)]{SilInt} in order to prove regularity of these
forms.
\end{remark}
\begin{proof}[Proof of Thm. \ref{09071617}]
We maintain the data and notation introduced in Rmk. \ref{09211852}, in addition we let $(U,\phi)$ be an oriented chart
of $M$ and $h\in\mc{A}(M)$ and $k\in\mc{A}(\partial M)$ be such that
\begin{equation}
\label{09211256}
\begin{cases}
\mr{supp}(h)\subseteq U;
\\
\mr{supp}(k)\subseteq U\cap\partial M.
\end{cases}
\end{equation}
Let $\{\sigma_{n}\}_{n\in\N}$ be an $E$-sequence of bounded sets and $n\in\N$,
let $\delta\in\{\sigma_{n},\sigma(R)\}$, let $R^{\delta}$ denote $R_{\delta}\up G_{\delta}$ and let $\uppsi\in\mc{N}$.
By Rmk. \ref{09071645} and Rmk. \ref{09071646} we have that
$[\omega,\sigma_{n},F]\in\mr{Alt}^{N-1}(M,\lr{B(G_{\sigma_{n}})}{\sigma(B(G_{\sigma_{n}})),\mc{N}_{\sigma_{n}}})$
so by \cite[Prp. 3.21, Prp. 3.24, Prp. 3.30]{silscess}, \eqref{09211256}, since the unique element of a smooth
partition of the unity subordinated to the open covering $\{U\}$ of $U$ equals $1$ when evaluated on $U$,
and finally we have
\begin{equation}
\begin{aligned}
\label{09080638a}
\int\uppsi_{\times}(hd[\omega,\sigma_{n},F])
&=
\int hd(\uppsi_{\times}[\omega,\sigma_{n},F])
\\
&=
\gamma_{\phi}\int
(\imath_{U}^{M}\circ\phi^{-1})^{\ast}(h)
(\imath_{U}^{M}\circ\phi^{-1})^{\times}
(d\uppsi_{\times}[\omega,\sigma_{n},F]).
\end{aligned}
\end{equation}
Next 
\begin{equation}
\begin{aligned}
\label{09080638b}
(\imath_{U}^{M}\circ\phi^{-1})^{\times}d\uppsi_{\times}[\omega,\sigma_{n},F]
&=
(\phi^{-1})^{\times}(\imath_{U}^{M})^{\times}d\uppsi_{\times}[\omega,\sigma_{n},F]
\\
&=
\uppsi_{\times}d(\phi^{-1})^{\times}(\imath_{U}^{M})^{\times}[\omega,\sigma_{n},F]
\\
&=
\uppsi_{\times}d(\phi^{-1})^{\times}[\omega,\phi,\sigma_{n},F];
\end{aligned}
\end{equation}
where the second equality follows by \cite[Prp. 3.21, Prp. 3.24, Prp. 3.30]{silscess},
the third one by Def. \ref{09211309} applied to any atlas containing $(U,\phi)$.
Now by definition 
\begin{equation}
\label{09081215}
[\omega,\phi,\delta,F]
=\sum_{i=1}^{N}\left(\zeta_{F,\omega_{i}^{\phi}}^{R^{\delta}}\circ\phi\right)
\otimes(dx_{1}^{\phi}\wedge\dots\widehat{dx_{i}^{\phi}}\wedge\dots dx_{N}^{\phi});
\end{equation}
thus
\begin{equation}
\begin{aligned}
\label{09080638c}
d(\phi^{-1})^{\times}([\omega,\phi,\sigma_{n},F]))
&=
\sum_{i=1}^{N}(-1)^{i-1}
\frac{\partial\zeta_{F,\omega_{i}^{\phi}}^{R^{\sigma_{n}}}}{\partial e_{i}}
\otimes(dx_{1}^{\mr{Id}_{\phi(U)}}\wedge\dots dx_{i}^{\mr{Id}_{\phi(U)}}\wedge\dots dx_{N}^{\mr{Id}_{\phi(U)}})
\\
&=
\sum_{i=1}^{N}(-1)^{i-1}
R^{\sigma_{n}}\frac{\partial\omega_{i}^{\phi}}{\partial e_{i}}\zeta_{\frac{dF}{d\lambda},\omega_{i}^{\phi}}^{R^{\sigma_{n}}}
\otimes(dx_{1}^{\mr{Id}_{\phi(U)}}\wedge\dots dx_{i}^{\mr{Id}_{\phi(U)}}\wedge\dots dx_{N}^{\mr{Id}_{\phi(U)}});
\end{aligned}
\end{equation}
where the second equality follows since Cor. \ref{09051744}.
Next $R^{\sigma_{n}}$ is norm continuous, thus by the end of Rmk. \ref{09071645} we have that 
\begin{equation}
\label{09080638d}
\int
(\imath_{U}^{M}\circ\phi^{-1})^{\ast}(h)
R^{\sigma_{n}}
\frac{\partial\omega_{i}^{\phi}}{\partial e_{i}}\zeta_{\frac{dF}{d\lambda},\omega_{i}^{\phi}}^{R^{\sigma_{n}}}
d\lambda_{\phi(U)}
=
R^{\sigma_{n}}
\int
(\imath_{U}^{M}\circ\phi^{-1})^{\ast}(h)
\frac{\partial\omega_{i}^{\phi}}{\partial e_{i}}\zeta_{\frac{dF}{d\lambda},\omega_{i}^{\phi}}^{R^{\sigma_{n}}}
d\lambda_{\phi(U)}
\in B(G_{\sigma_{n}});
\end{equation}
the integrals being w.r.t. the norm topology on $B(G_{\sigma_{n}})$ then also w.r.t. the
$\lr{B(G_{\sigma_{n}})}{\sigma(B(G_{\sigma_{n}}),\mc{N}_{\sigma_{n}})}$ topology since Rmk. \ref{09080624}.
Now \eqref{09080638a},\,\eqref{09080638b},\,\eqref{09080638c} and \eqref{09080638d} yield
\begin{equation}
\label{09080645}
\begin{aligned}
\int hd[\omega,\sigma_{n},F]
&=
R^{\sigma_{n}}
\gamma_{\phi}
\int
(\imath_{U}^{M}\circ\phi^{-1})^{\ast}(h)
\sum_{i=1}^{N}(-1)^{i-1}\frac{\partial\omega_{i}^{\phi}}{\partial e_{i}}
\zeta_{\frac{dF}{d\lambda},\omega_{i}^{\phi}}^{R^{\sigma_{n}}} d\lambda_{\phi(U)}
\\
&=
R^{\sigma_{n}}
\int
h\sum_{i=1}^{N}(-1)^{i-1}\mf{d}_{i}(\omega)\cdot[\mf{n}_{i}(\omega),\sigma_{n},\frac{dF}{d\lambda}];
\end{aligned}
\end{equation}
integrals w.r.t. the $\lr{B(G_{\sigma_{n}})}{\sigma(B(G_{\sigma_{n}}),\mc{N}_{\sigma_{n}})}$ topology, where the second equality
follows by the next equality obtained by direct calculation 
\begin{equation}
\label{09241810pre}
\int
h\sum_{i=1}^{N}(-1)^{i-1}\mf{d}_{i}(\omega)\cdot[\mf{n}_{i}(\omega),\delta,\frac{dF}{d\lambda}]
=
\gamma_{\phi}
\int
(\imath_{U}^{M}\circ\phi^{-1})^{\ast}(h)
\sum_{i=1}^{N}(-1)^{i-1}\frac{\partial\omega_{i}^{\phi}}{\partial e_{i}}
\zeta_{\frac{dF}{d\lambda},\omega_{i}^{\phi}}^{R^{\delta}} d\lambda_{\phi(U)}.
\end{equation}
Now by \eqref{09080645} applied to $(U,\phi)=(U_{\alpha},\phi_{\alpha})$ and $h=\psi_{\alpha}$ for every
$\alpha\in D$ and since \cite[Cor.3.58]{silscess}  we obtain
\begin{equation}
\label{09211741}
\begin{aligned}
\int d[\omega,\sigma_{n},F]
&=
R^{\sigma_{n}}
\int
\sum_{i=1}^{N}(-1)^{i-1}\mf{d}_{i}(\omega)\cdot[\mf{n}_{i}(\omega),\sigma_{n},\frac{dF}{d\lambda}]
\\
&=
R^{\sigma_{n}}
\int
\sum_{\alpha\in D}
\gamma_{\phi_{\alpha}}
(\imath_{U_{\alpha}}^{M}\circ\phi_{\alpha}^{-1})^{\ast}(\psi_{\alpha})
\sum_{i=1}^{N}(-1)^{i-1}\frac{\partial\omega_{i}^{\phi_{\alpha}}}{\partial e_{i}}
\zeta_{\frac{dF}{d\lambda},\omega_{i}^{\phi_{\alpha}}}^{R^{\sigma_{n}}} d\lambda_{\phi_{\alpha}(U_{\alpha})};
\end{aligned}
\end{equation}
where all the three integrals are w.r.t.the $\lr{B(G_{\sigma_{n}})}{\sigma(B(G_{\sigma_{n}}),\mc{N}_{\sigma_{n}})}$ topology.
Now if $\partial M=\emptyset$ the statement follows by the above equality, \cite[Thm.3.59]{silscess} and by our
Extension Thm. \ref{18051958ta}. Thus in what follows assume in addition that $\partial M\neq\emptyset$ and that
$(U,\phi)$ is a boundary chart, therefore  $(U,\phi^{\partial M})$ is a chart of $\partial M$ 
such that $\gamma_{\phi^{\partial M}}=\gamma_{\phi}$. Next since the unique element of a smooth
partition of the unity subordinated to the open covering $\{U\cap\partial M\}$, w.r.t. the topological space $\partial M$,
of $U\cap\partial M$ equals $1$ when evaluated on $U\cap\partial M$, we have by 
\cite[Prp. 3.21, Prp. 3.24, Prp. 3.30]{silscess}, \eqref{09211256}, and $\gamma_{\phi^{\partial M}}=\gamma_{\phi}$
\begin{equation}
\label{09201247}  
\begin{aligned}
\int
\uppsi_{\times}\left(k(\imath_{\partial M}^{M})^{\times}[\omega,\delta,F]\right)
&=
\int
k\uppsi_{\times}(\imath_{\partial M}^{M})^{\times}[\omega,\delta,F]
\\
&=
\gamma_{\phi}
\int
\left((\imath_{U\cap\partial M}^{\partial M}\circ(\phi^{\partial M})^{-1})^{\ast}k\right)
(\imath_{U\cap\partial M}^{\partial M}\circ(\phi^{\partial M})^{-1})^{\times}
\uppsi_{\times}(\imath_{\partial M}^{M})^{\times}[\omega,\delta,F]
\\
&=
\gamma_{\phi}
\int
\left((\imath_{U\cap\partial M}^{\partial M}\circ(\phi^{\partial M})^{-1})^{\ast}k\right)
\uppsi_{\times}
((\phi^{\partial M})^{-1})^{\times}(\imath_{U\cap\partial M}^{\partial M})^{\times}(\imath_{\partial M}^{M})^{\times}[\omega,\delta,F]
\\
&=
\gamma_{\phi}
\int
\left((\imath_{U\cap\partial M}^{\partial M}\circ(\phi^{\partial M})^{-1})^{\ast}k\right)
\uppsi_{\times}
((\phi^{\partial M})^{-1})^{\times}(\imath_{U\cap\partial M}^{M})^{\times}[\omega,\delta,F]
\\
&=
\gamma_{\phi}
\int
\left((\imath_{U\cap\partial M}^{\partial M}\circ(\phi^{\partial M})^{-1})^{\ast}k\right)
\uppsi_{\times}
((\phi^{\partial M})^{-1})^{\times}(\imath_{U\cap\partial M}^{U})^{\times}(\imath_{U}^{M})^{\times}[\omega,\delta,F].
\end{aligned}
\end{equation}
Next by \eqref{09081215} and since $(\imath_{U\cap\partial M}^{U})^{\ast}(dx_{N}^{\phi})=\ze$, we obtain 
\begin{equation*}
\begin{aligned}
(\imath_{U\cap\partial M}^{U})^{\times}(\imath_{U}^{M})^{\times}[\omega,\delta,F]
&=
(\imath_{U\cap\partial M}^{U})^{\times}[\omega,\phi,\delta,F]
\\
&=
(\zeta_{F,\omega_{N}^{\phi}}^{R^{\delta}}\circ\phi\circ\imath_{U\cap\partial M}^{U})
\otimes
\bigwedge_{s=1}^{N-1}
(\imath_{U\cap\partial M}^{U})^{\ast}
(dx_{s}^{\phi})
\\
&=
(\zeta_{F,\omega_{N}^{\phi}}^{R^{\delta}}\circ\phi\circ\imath_{U\cap\partial M}^{U})
\otimes
\bigwedge_{s=1}^{N-1}
dx_{s}^{\phi^{\partial M}};
\end{aligned}
\end{equation*}
and by letting $Z\coloneqq\phi^{\partial M}(U\cap\partial M)$
\begin{equation}
\label{09201244}
((\phi^{\partial M})^{-1})^{\times}(\imath_{U\cap\partial M}^{U})^{\times}(\imath_{U}^{M})^{\times}[\omega,\delta,F]
=
(\zeta_{F,\omega_{N}^{\phi}}^{R^{\delta}}\circ\phi\circ\imath_{U\cap\partial M}^{U}\circ(\phi^{\partial M})^{-1})
\otimes\bigwedge_{s=1}^{N-1}dx_{s}^{\mr{Id}_{Z}}.
\end{equation}
Next
\begin{equation*}
\mf{i}_{\R^{N-1}}^{\R^{N}}(Z)=\phi(U\cap\partial M)\subseteq\partial\mathbb{H}^{N};
\end{equation*}
by definition of boundary chart of $M$.
Define $\mr{P}_{[N]}\coloneqq\mf{i}_{\R^{N-1}}^{\R^{N}}\circ\mr{P}^{[N]}$, thus 
by letting $\Pr(\R^{N})$ be the set of projectors of $\R^{N}$, we have  
\begin{equation}
\label{09201239}
\begin{cases}
\mr{P}_{[N]}\in\Pr(\R^{N}),
\\
\partial\mathbb{H}^{N}=\mr{P}_{[N]}(\R^{N});
\end{cases}
\end{equation}
moreover by definition of $\phi^{\partial M}$ we have 
\begin{equation}
\label{09201146}
(\forall x\in Z)
\left(\mr{P}_{[N]}\left((\phi\circ\imath_{U\cap\partial M}^{U}\circ(\phi^{\partial M})^{-1})(x)\right)
=
\mf{i}_{\R^{N-1}}^{\R^{N}}(x)\right).
\end{equation}
Now $(\phi\circ\imath_{U\cap\partial M}^{U}\circ(\phi^{\partial M})^{-1})(x)\in\partial\mathbb{H}^{N}$
since $\phi$ is a boundary chart of $M$, therefore by \eqref{09201146} and \eqref{09201239} we obtain 
\begin{equation*}
\phi\circ\imath_{U\cap\partial M}^{U}\circ(\phi^{\partial M})^{-1}
=
\imath_{\phi(U\cap\partial M)}^{\phi(U)}\circ(\mf{i}_{\R^{N-1}}^{\R^{N}}\circ\imath_{Z}^{\R^{N-1}})_{\natural}.
\end{equation*}
Therefore by \eqref{09201244}
\begin{equation*}
((\phi^{\partial M})^{-1})^{\times}(\imath_{U\cap\partial M}^{U})^{\times}(\imath_{U}^{M})^{\times}[\omega,\delta,F]
=
\left(\zeta_{F,\omega_{N}^{\phi}}^{R^{\delta}}\circ\imath_{\phi(U\cap\partial M)}^{\phi(U)}
\circ(\mf{i}_{\R^{N-1}}^{\R^{N}}\circ\imath_{Z}^{\R^{N-1}})_{\natural}\right)
\otimes
\bigwedge_{s=1}^{N-1}
dx_{s}^{\mr{Id}_{Z}};
\end{equation*}
hence by \eqref{09201247} we obtain
\begin{equation}
\label{09081235}
\int
k(\imath_{\partial M}^{M})^{\times}[\omega,\delta,F]
=
\gamma_{\phi}
\int
((\imath_{U\cap\partial M}^{\partial M}\circ(\phi^{\partial M})^{-1})^{\ast}k)
\left(\zeta_{F,\omega_{N}^{\phi}}^{R^{\delta}}
\circ\imath_{\phi(U\cap\partial M)}^{\phi(U)}
\circ(\mf{i}_{\R^{N-1}}^{\R^{N}}\circ\imath_{Z}^{\R^{N-1}})_{\natural}\right)\,d\lambda_{Z};
\end{equation}
where the integrals are w.r.t. the $\lr{B(G_{\delta})}{\sigma(B(G_{\delta}),\mc{N}_{\delta})}$ topology.
Now by \eqref{09081235} applied to $(U,\phi)=(U_{\alpha},\phi_{\alpha})$ and $k=k_{\alpha}$ for every
$\alpha\in D$ and since \cite[Cor.3.58]{silscess} we obtain by letting
$Z_{\alpha}\coloneqq\phi_{\alpha}^{\partial M}(U_{\alpha}\cap\partial M)$
\begin{equation}
\label{09211740}
\int(\imath_{\partial M}^{M})^{\times}[\omega,\delta,F]
=
\int
\sum_{\alpha\in D}
\gamma_{\phi_{\alpha}}
((\imath_{U_{\alpha}\cap\partial M}^{\partial M}\circ(\phi_{\alpha}^{\partial M})^{-1})^{\ast}k_{\alpha})
\left(\zeta_{F,\omega_{N}^{\phi_{\alpha}}}^{R^{\delta}}\circ\imath_{\phi_{\alpha}(U_{\alpha}\cap\partial M)}^{\phi_{\alpha}(U_{\alpha})}
\circ(\mf{i}_{\R^{N-1}}^{\R^{N}}\circ\imath_{Z_{\alpha}}^{\R^{N-1}})_{\natural}\right)\,
d\lambda_{Z_{\alpha}}.
\end{equation}
Next by \eqref{09241810pre} applied to $(U,\phi)=(U_{\alpha},\phi_{\alpha})$ and $h=\psi_{\alpha}$ for every
$\alpha\in D$ and since \cite[Cor.3.58]{silscess} we obtain 
\begin{equation}
\label{09241810}
\int
\sum_{i=1}^{N}(-1)^{i-1}\mf{d}_{i}(\omega)\cdot[\mf{n}_{i}(\omega),\delta,\frac{dF}{d\lambda}]
=
\int
\sum_{\alpha\in D}
\gamma_{\phi_{\alpha}}
(\imath_{U_{\alpha}}^{M}\circ\phi_{\alpha}^{-1})^{\ast}(\psi_{\alpha})
\sum_{i=1}^{N}(-1)^{i-1}\frac{\partial\omega_{i}^{\phi_{\alpha}}}{\partial e_{i}}
\zeta_{\frac{dF}{d\lambda},\omega_{i}^{\phi_{\alpha}}}^{R^{\delta}} d\lambda_{\phi_{\alpha}(U_{\alpha})}.
\end{equation}
Now since \cite[Thm.3.59]{silscess} applied to the form $[\omega,\sigma_{n},F]$, since \eqref{09211741} and
since \eqref{09211740} applied to $\delta=\sigma_{n}$ we obtain
\begin{multline*}
R^{\sigma_{n}}
\int
\sum_{\alpha\in D}
\gamma_{\phi_{\alpha}}
(\imath_{U_{\alpha}}^{M}\circ\phi_{\alpha}^{-1})^{\ast}(\psi_{\alpha})
\sum_{i=1}^{N}(-1)^{i-1}\frac{\partial\omega_{i}^{\phi_{\alpha}}}{\partial e_{i}}
\zeta_{\frac{dF}{d\lambda},\omega_{i}^{\phi_{\alpha}}}^{R^{\sigma_{n}}}\,d\lambda_{\phi_{\alpha}(U_{\alpha})}
=
\\
\int
\sum_{\alpha\in D}
\gamma_{\phi_{\alpha}}
((\imath_{U_{\alpha}\cap\partial M}^{\partial M}\circ(\phi_{\alpha}^{\partial M})^{-1})^{\ast}k_{\alpha})
\left(\zeta_{F,\omega_{N}^{\phi_{\alpha}}}^{R^{\sigma_{n}}}\circ\imath_{\phi_{\alpha}(U_{\alpha}\cap\partial M)}^{\phi_{\alpha}(U_{\alpha})}
\circ(\mf{i}_{\R^{N-1}}^{\R^{N}}\circ\imath_{Z_{\alpha}}^{\R^{N-1}})_{\natural}\right)\,d\lambda_{Z_{\alpha}}.
\end{multline*}
Now we can employ our Extension Thm. \ref{18051958ta} to the above sequence of equality to obtain 
\begin{multline*}
R
\int
\sum_{\alpha\in D}
\gamma_{\phi_{\alpha}}
(\imath_{U_{\alpha}}^{M}\circ\phi_{\alpha}^{-1})^{\ast}(\psi_{\alpha})
\sum_{i=1}^{N}(-1)^{i-1}\frac{\partial\omega_{i}^{\phi_{\alpha}}}{\partial e_{i}}
\zeta_{\frac{dF}{d\lambda},\omega_{i}^{\phi_{\alpha}}}^{R}\,d\lambda_{\phi_{\alpha}(U_{\alpha})} 
=
\\
\int
\sum_{\alpha\in D}
\gamma_{\phi_{\alpha}}
\left(\imath_{U_{\alpha}\cap\partial M}^{\partial M}\circ(\phi_{\alpha}^{\partial M})^{-1}\right)^{\ast}(k_{\alpha})
\left(\zeta_{F,\omega_{N}^{\phi_{\alpha}}}^{R}\circ\imath_{\phi_{\alpha}(U_{\alpha}\cap\partial M)}^{\phi_{\alpha}(U_{\alpha})}
\circ(\mf{i}_{\R^{N-1}}^{\R^{N}}\circ\imath_{\phi_{\alpha}^{\partial M}(U_{\alpha}\cap\partial M)}^{\R^{N-1}})_{\natural}\right)\,
d\lambda_{\phi_{\alpha}^{\partial M}(U_{\alpha}\cap\partial M)};
\end{multline*}
and the statement follows since \eqref{09241810} and \eqref{09211740} applied to $\delta=\sigma(R)$.
\end{proof}
\begin{corollary}
\label{09261423}
Let $M$ be oriented with boundary and $\omega\in\mr{Alt}_{c}^{N-1}(M)$.
Let $V$ be an open neighbourhood of $\sigma(R)$ such that $\R\cdot V\subseteq V$ and $F:V\to\C$ be analytic. 
Assume that there exists a finite collection $\mc{U}=\{(U_{\alpha},\phi_{\alpha})\}_{\alpha\in D}$
of oriented charts of $M$ such that $\{U_{\alpha}\}_{\alpha\in D}$ is a covering of the support of $\omega$ and
\begin{enumerate}
\item
$\widetilde{F}_{t}\in\mf{L}_{E}^{\infty}(\sigma(R))$ for every $t\in\R$, and for all $\alpha\in D$ such that $\phi_{\alpha}$
is a boundary chart, the map
\begin{equation*}
\uppsi\circ\zeta_{F,\omega_{N}^{\phi_{\alpha}}}^{R}\circ\imath_{\phi_{\alpha}(U_{\alpha}\cap\partial M)}^{\phi_{\alpha}(U_{\alpha})}
\end{equation*}
is $\lambda_{\phi_{\alpha}(U_{\alpha}\cap\partial M)}$-measurable for every $\uppsi\in\mc{N}$;
and 
\begin{equation*}
\int^{\ast}
\|\cdot\|_{\infty}^{E}\circ\widetilde{F}_{\omega_{N}^{\phi_{\alpha}}\circ\imath_{\phi_{\alpha}(U_{\alpha}\cap\partial M)}^{\phi_{\alpha}(U_{\alpha})}}
\,d\lambda_{\phi_{\alpha}(U_{\alpha}\cap\partial M)}
<\infty;
\end{equation*}
\label{09261423hp1}
\item
$(\widetilde{\frac{dF}{d\lambda}})_{t}\in\mf{L}_{E}^{\infty}(\sigma(R))$ for every $t\in\R$, and for all $\alpha\in D$ and 
$i\in[1,N]\cap\Z$, the map 
\begin{equation*}
\uppsi\circ\zeta_{\frac{dF}{d\lambda},\omega_{i}^{\phi_{\alpha}}}^{R}
\end{equation*}
is $\lambda_{\phi_{\alpha}(U_{\alpha})}$-measurable for every $\uppsi\in\mc{N}$;
and 
\begin{equation*}
\int^{\ast}
\|\cdot\|_{\infty}^{E}\circ\left(\widetilde{\frac{dF}{d\lambda}}\right)_{\omega_{i}^{\phi_{\alpha}}}
\,d\lambda_{\phi_{\alpha}(U_{\alpha})}<\infty.
\end{equation*}
\label{09261423hp2}
\end{enumerate}
Thus the statement of Thm. \ref{09071617} holds true.
Moreover if in addition $\mc{N}$ is an $E$-appropriate set with the isometric duality property and 
$C\doteqdot\sup_{\sigma\in\mc{B}(\C)}\|E(\sigma)\|$, then we obtain the following estimates
\begin{equation*}
\left\|\int\zeta_{F,\omega_{N}^{\phi_{\alpha}}}^{R}\circ\imath_{\phi_{\alpha}(U_{\alpha}\cap\partial M)}^{\phi_{\alpha}(U_{\alpha})}\,
d\lambda_{\phi_{\alpha}(U_{\alpha}\cap\partial M)}\right\|
\leq
4C\int^{\ast}
\|\cdot\|_{\infty}^{E}\circ
\widetilde{F}_{\omega_{N}^{\phi_{\alpha}}\circ\imath_{\phi_{\alpha}(U_{\alpha}\cap\partial M)}^{\phi_{\alpha}(U_{\alpha})}}
\,d\lambda_{\phi_{\alpha}(U_{\alpha}\cap\partial M)},
\end{equation*}
and
\begin{equation*}
\left\|\int\zeta_{\frac{dF}{d\lambda},\omega_{i}^{\phi_{\alpha}}}^{R}\,d\lambda_{\phi_{\alpha}(U_{\alpha})}\right\|
\leq
4C\int^{\ast}\|\cdot\|_{\infty}^{E}\circ\left(\widetilde{\frac{dF}{d\lambda}}\right)_{\omega_{i}^{\phi_{\alpha}}}
\,d\lambda_{\phi_{\alpha}(U_{\alpha})}.
\end{equation*}
\end{corollary}
\begin{remark}
By letting $\mr{Bor}(\C)$ be the set of complex valued Borelian maps on $\C$, we have  
\begin{equation*}
\begin{cases}
\widetilde{F}_{\omega_{N}^{\phi_{\alpha}}\circ\imath_{\phi_{\alpha}(U_{\alpha}\cap\partial M)}^{\phi_{\alpha}(U_{\alpha})}}:
\phi_{\alpha}(U_{\alpha}\cap\partial M)\to\mr{Bor}(\C)
\\
x\mapsto\widetilde{F}_{(\omega_{N}^{\phi_{\alpha}}\circ\imath_{\phi_{\alpha}(U_{\alpha}\cap\partial M)}^{\phi_{\alpha}(U_{\alpha})})(x)},
\end{cases}
\end{equation*}
and
\begin{equation*}
\begin{cases}
\left(\widetilde{\frac{dF}{d\lambda}}\right)_{\omega_{i}^{\phi_{\alpha}}}:
\phi_{\alpha}(U_{\alpha})\to\mr{Bor}(\C) 
\\
x\mapsto\left(\widetilde{\frac{dF}{d\lambda}}\right)_{\omega_{i}^{\phi_{\alpha}}(x)};
\end{cases}
\end{equation*}
where we recall that $L_{t}:\lambda\mapsto L(t\lambda)$ for any $L\in\mr{Bor}(\C)$ and any $t\in\R$.
Therefore the upper integrals in hypotheses \eqref{09261423hp1} and \eqref{09261423hp2} are well-set.
\end{remark}  
\begin{proof}
By \cite[(1.42)]{SilInt}, hypotheses, and \cite[Thm. 18.2.11(c)]{ds} we obtain that
\begin{equation*}
\|\cdot\|\circ\zeta_{F,\omega_{N}^{\phi_{\alpha}}}^{R}\circ\imath_{\phi_{\alpha}(U_{\alpha}\cap\partial M)}^{\phi_{\alpha}(U_{\alpha})}
\in\mf{F}_{1}(\phi_{\alpha}(U_{\alpha}\cap\partial M),\lambda_{\phi_{\alpha}(U_{\alpha}\cap\partial M)}),
\end{equation*}
and 
\begin{equation*}
\|\cdot\|\circ\zeta_{\frac{dF}{d\lambda},\omega_{i}^{\phi_{\alpha}}}^{R}
\in\mf{F}_{1}(\phi_{\alpha}(U_{\alpha}),\lambda_{\phi_{\alpha}(U_{\alpha})}).
\end{equation*}
Then the hypotheses of Thm. \ref{09071617} are satisfied by \cite[footnote $1$ pg. 39]{SilInt} and by \cite[Thm. 2.2]{SilInt}
and the first part of the statement follows. The estimates in the statement follow by the estimate in
\cite[Thm. 2.2]{SilInt} and by \cite[Thm. 18.2.11(c)]{ds}.
\end{proof}
\begin{corollary}
\label{09262234}
The statement of Cor. \ref{09261423} holds if $G$ is a complex Hilbert space and $\mc{N}$ is replaced by $\mc{N}_{pd}(G)$.
\end{corollary}
\begin{proof}
By the end of \cite[Rmk. 2.12]{SilInt} and by Cor. \ref{09261423}.
\end{proof}
\begin{corollary}
\label{09262238}
The statement of Cor. \ref{09261423} holds if $G$ is reflexive and $\mc{N}$ is replaced by
$\mc{N}_{st}(G)$.
\end{corollary}
\begin{proof}
By employing \cite[Cor. 2.6]{SilInt} instead of \cite[Thm. 2.2]{SilInt} 
the proof runs exactly as the one in Cor. \ref{09261423}.
\end{proof}


\bibliography{references}
\bibliographystyle{amsplain}

\end{document}